\documentclass[11pt,a4paper]{article}

\usepackage{amsmath, amssymb, amsthm}
\usepackage{geometry}
\usepackage{color}
\usepackage[numbers,sort,compress]{natbib}

\newtheorem{theorem}{Theorem}[section]
\newtheorem{lemma}[theorem]{Lemma}
\newtheorem{corollary}[theorem]{Corollary}
\newtheorem{proposition}[theorem]{Proposition}

\theoremstyle{definition}
\newtheorem{definition}[theorem]{Definition}
\newtheorem{remark}[theorem]{Remark}
\newtheorem{example}[theorem]{Example}

\newcommand{\F}{\mathbb{F}}
\newcommand{\N}{\mathcal{N}}

\newcommand{\sfp}{\operatorname{sfp}}
\newcommand{\Res}{\operatorname{Res}}

\newcommand{\col}{\operatorname{col}}
\newcommand{\rad}{\operatorname{rad}}
\newcommand{\diag}{\operatorname{diag}}
\newcommand{\sqrtp}{\operatorname{sqrt}}
\newcommand{\nullity}{\operatorname{nullity}}
\newcommand{\tr}{\operatorname{tr}}
\newcommand{\adj}{\operatorname{adj}}

\title{Factorization of invariant polynomials and generalized spectral characterizations of graphs}
\author{Wei Wang$^{{\rm a}}$\thanks{Corresponding author.\\	
		\hspace*{1.8em}E-mail addresses: wangwei.math@gmail.com (W.~Wang), tangquanyu827@gmail.com (Q.~Tang),\\
		wang\_weiw@xjtu.edu.cn (W.~Wang).} \quad Quanyu Tang$^{{\rm b}}$ \quad Wei Wang$^{{\rm b}}$\\
{\footnotesize$^{\rm a}$School of Mathematics, Physics and Finance, Anhui Polytechnic University, Wuhu 241000, P. R. China}\\
{\footnotesize$^{\rm b}$School of Mathematics and Statistics, Xi'an Jiaotong University, Xi'an 710049, P. R. China}
}
\date{}
\begin{document}
		\maketitle
\begin{abstract}
	The problem of characterizing graphs by their generalized spectra has received significant attention in recent years. This paper provides a complete proof of a conjecture proposed by Wang, Wang, and Zhu (European J. Combin., 2023), which asserts that the square-root polynomial of the invariant polynomial $\Phi_p(G;x) \in \F_p[x]$ can replace its square-free part to yield a more effective criterion for a graph to be determined by its generalized spectrum (DGS). A key ingredient of our proof is a novel algebraic factorization: we show that the polynomial $\Phi_p(G;x)$ is the product of the characteristic polynomials of the adjacency operator restricted to the left null space of the walk matrix and its radical, respectively. Based on this refined DGS-criterion, a broad family of DGS-graphs is constructed via rooted products, significantly generalizing the recent result of Wang, Shen, and Mao (Discrete Appl. Math., 2026).
	\end{abstract}
	

	\section{Introduction}
	
	Let $G$ be a simple graph on $n$ vertices with adjacency matrix $A$. We often identify a graph $G$ with its adjacency matrix $A$.  The \emph{spectrum}
	of a graph $G$, denoted by $\sigma(G)$, is the multiset of eigenvalues of $A$. The \emph{generalized
	spectrum} of $G$ is the ordered pair $(\sigma(G), \sigma(\overline{G}))$, where $\overline{G}$ is the complement of $G$. Two
	graphs are \emph{generalized cospectral} if they share the same generalized spectrum; a graph
	$G$ is \emph{determined by its generalized spectrum} (DGS) if any graph generalized cospectral with $G$ is isomorphic to $G$.
	
	The \emph{walk matrix} of $G$ is defined as 
	\begin{equation*}
	W =W(G)= [e, Ae, \dots, A^{n-1}e], 
	\end{equation*}where $e$ is the all-ones vector. This kind of matrices has attracted increasing attention in recent years, as many important properties of graphs, including the DGS-property, are closely related to the corresponding matrices; see, e.g. \cite{godsil1981JCTB,liu2022JAC,guo2025AAM,wang2023EJC,rowlinson2007AADM}.
	
	It is known \cite{wang2013EJC} that for an $n$-vertex graph $G$, the determinant of $W(G)$ is always a multiple of $2^{\lfloor n/2\rfloor}$, i.e.,  $2^{-\lfloor n/2\rfloor}\det W(G)\in \mathbb{Z}$. An integer $b$ is called \emph{square-free} if $p^2\nmid b$ for any prime $p$. In particular, the number $\pm 1$ is square-free, but 0 is not.  Wang \cite{wang2013EJC,wang2017JCTB} gave the following arithmetic criterion for a graph to be DGS.
	\begin{theorem}[\cite{wang2013EJC,wang2017JCTB}]\label{wt}
	Let $G$ be an $n$-vertex graph.	If $2^{-\lfloor n/2\rfloor}\det W(G)$ is odd and square-free, then $G$ is DGS.
	\end{theorem}
	The original proof of Theorem \ref{wt} was somewhat involved; however, a considerably simpler proof was subsequently developed by Qiu et al.~\cite{qiu2023DM}, based on Smith normal forms. Motivated by the argument in \cite{qiu2023DM}, Wang et al.~\cite{wang2023EJC} considered a larger family of graphs and obtained an improvement of Theorem \ref{wt}.  
	
	Recall that for a non-singular integer matrix $M$ of order $n$, there exist unimodular matrices $U,V$ such that $UMV=\diag[d_1,\ldots,d_n]$, where $d_i$'s are positive integers with $d_1\mid d_2\cdots\mid d_n$. The diagonal matrix $\diag(d_1,\ldots,d_n)$ is called the \emph{Smith normal form} of $M$; the $i$-th diagonal entry $d_i$ is called the $i$-th \emph{invariant factor} of $M$. We note that the condition of Theorem \ref{wt} can be restated using the Smith normal forms.
	
		\begin{theorem}[\cite{wang2013EJC,wang2017JCTB}]\label{wt2}
		Let $G$ be an $n$-vertex graph.	If $W(G)$ has the Smith normal form
		\begin{equation*}
			\diag[\underbrace{1,\ldots,1}_{\lceil\frac{n}{2}\rceil},\underbrace{2,\ldots,2,2b}_{\lfloor\frac{n}{2}\rfloor}],
		\end{equation*}
		 \text{where $b$ is odd and square-free}, then $G$ is DGS.
	\end{theorem}
	The following polynomial $\Phi_p(G;x)$ introduced in \cite{wang2023EJC} is crucial for their improvement. They only considered the case $p$ is odd. Here we allow $p=2$.  For a square matrix $M$, we use $\chi(M;x)$ to denote the characteristic polynomial of $M$, i.e., $\chi(M;x)=\det(xI-M)$. We use $J$ to denote the all-ones matrix.
\begin{definition}[\cite{wang2023EJC}]\label{Phip}\normalfont
	Let $p$ be a prime and $G$ be a graph with adjacency matrix $A$. We define	
	\begin{equation*}\Phi_p(G;x)=\gcd(\chi(A;x),\chi(A+J;x))\in \mathbb{F}_p[x],
	\end{equation*}
	where  the greatest common divisor (gcd) is taken over the field $\mathbb{F}_p=\mathbb{Z}/p\mathbb{Z}$.
\end{definition}
\begin{remark}
	The polynomial $\Phi_p(G;x)$ is invariant under generalized cospectrality. That is, if $G$ and $H$ are generalized cospectral, then $\Phi_p(G;x)=\Phi_p(H;x)$ for any prime $p$.
\end{remark}
Let $f$ be a monic polynomial in $\mathbb{F}_p[x]$ and $f = \prod_{1\le i\le r}f_i^{e_i}$
be the irreducible factorization of $f$, with distinct monic irreducible polynomials $f_1,f_2,\ldots, f_r$ and positive integers $e_1,e_2,\ldots, e_r$. The \emph{square-free part} \cite[p.~394]{gathen} of $f$, denoted by $\sfp (f)$,  is $\prod_{1\le i\le r}f_i$; and the \emph{square root} of $f$, denoted by $\sqrtp(f)$, is $\prod_{1\le i\le r}f_i^{\lceil\frac{e_i}{2}\rceil}$. For an integer matrix $M$ and a prime $p$, we use $\nullity_p M$ to denote the nullity of $M$ over $\F_p$.

	\begin{theorem}[\cite{wang2023EJC}]\label{main0}
		Let $G$ be an $n$-vertex graph and $d_n$ be the last invariant factor of $W=W(G)$. Suppose that $d_n$ is square-free. If for each odd prime factor $p$ of $d_n$,
		\begin{equation}\label{keyequ1}
			\deg\sfp \Phi_p(G;x)= \nullity_p W,
		\end{equation} then $G$ is DGS.
	\end{theorem}
	They conjectured that the polynomial $\sfp \Phi_p(G;x)$  can be replaced by $\sqrtp \Phi_p(G;x)$. 	The main aim of this paper is to verify this conjecture, which we state as the following theorem
	\begin{theorem}\label{main}
	Let $G$ be an $n$-vertex graph and $d_n$ be the last invariant factor of $W=W(G)$. Suppose that $d_n$ is square-free. If for each odd prime factor $p$ of $d_n$,
	\begin{equation}\label{keyequ2}
		\deg\sqrtp \Phi_p(G;x)= \nullity_p W,
	\end{equation} then $G$ is DGS.
	\end{theorem}
	Note that, by definition, $\deg \sfp \Phi_p(G;x)\le \deg \sqrtp \Phi_p(G;x)$. We shall show that 	$\deg\sqrtp \Phi_p(G;x)$ is always upper bounded by $\nullity_p W$ (see Corollary \ref{sqrtmid}). It follows that
	\begin{equation*}
		\deg\sfp \Phi_p(G;x)\le \deg \sqrtp \Phi_p(G;x)\le \nullity_p W.
	\end{equation*}
	This means that Eq.~\eqref{keyequ1} implies Eq.~\eqref{keyequ2}; but the converse is not true in general as shown by later examples. Thus, Theorem \ref{main} is indeed an improvement of Theorem \ref{main0}.
	
	A crucial step for the proof of Theorem \ref{main} is the following factorization of the polynomial $\Phi_p(G;x)$:
	\begin{equation}\label{fac}
		\Phi_p(G;x)=\chi(A|_{\N(W^\top)};x)\chi(A|_{\rad \N(W^\top)};x),
	\end{equation} 
	where $\rad \N(W^\top)\subset \F_p^n$ is the radical of the left null space $\N(W^\top)$, in symbols, 
	\begin{equation*}
\rad \N(W^\top)=\N(W^\top)\cap \left(\N(W^\top)\right)^\perp.
	\end{equation*}
It turns out that
	the polynomial
	\begin{equation}\label{mp}
	\frac{\chi(G;x)}{\sqrtp \Phi_p(G;x)}
	\end{equation}
	is an annihilating polynomial of $A$ with respect to the all-ones vector $e$. Theorem \ref{main} follows from this newly established fact combined with the original argument in \cite{wang2023EJC}.
	
	We remark that all our arguments hold for any prime $p$, including $p=2$. Notably, for the case $p=2$, we shall show that the current arguments agree with the method of Qiu et al.~\cite{qiu2023DM}. Explicitly, we shall show that the polynomial defined in Eq.~\eqref{mp} for the case $p=2$ is exactly the annihilating polynomial of $A$ given in \cite[Lemma 2.10]{qiu2023DM}.
	
	The remainder of this paper is organized as follows. In Section 2, we establish the 
	fundamental factorization of the polynomial $\Phi_p(G;x)$, which serves as the 
	algebraic basis for our study. Section 3 is devoted to the complete proof of 
	Theorem \ref{main}. In Section 4, we demonstrate that our new polynomial framework 
	is consistent with existing methods for the case $p=2$. Finally, in Section 5, 
	we apply Theorem \ref{main} to the study of the constructions of DGS-graphs via rooted product graphs. We introduce 
	a broad family of graphs, denoted by $\mathcal{G}$, and prove that the rooted 
	product operation preserves the DGS-property under certain conditions. This 
	significantly extends the construction methods recently reported in \cite{wang2026DAM}, 
	providing a more flexible framework for generating DGS graphs using rooted product operations.

	\section{A factorization of $\Phi_p(G;x)$}
	The main aim of this section is to establish Eq.~\eqref{fac}.	Throughout this section, we fix a prime $p$, and all matrices and polynomials are evaluated over the finite field $\F_p$. Let $G$ be an $n$-vertex graph with adjacency matrix $A$. Let $R(G;x)$, or simply $R(x)$, be the formal power series  in $x^{-1}$ defined by:
    \begin{equation*}
	R(x) = e^\top (xI - A)^{-1} e
	= \sum_{j=0}^\infty \frac{e^\top A^j e}{x^{j+1}}
	\in \F_p[[x^{-1}]].
    \end{equation*}
			
	\begin{lemma}[Matrix Determinant Lemma {\cite[Sec.~18.1]{harville2008}}] \label{mdl}
		Let $M$ be an  $n \times n$ matrix, and $u, v$ be $n$-dimensional column vectors. Then
		\begin{equation*}
			\det(M + u v^\top) = \det(M)  + v^\top (\adj M) u.
		\end{equation*}
	\end{lemma}
	\begin{definition}[\cite{wang2023EJC}]\label{pm}\normalfont
		The \emph{$p$-main polynomial} of a graph $G$, denoted by $m_p(G;x)$ or $m_p(x)$, is  the monic polynomial $f\in \mathbb{F}_p[x]$ of
		smallest degree such that $f(A)e = 0$.
	\end{definition}
	\begin{lemma}\label{rsq}
	$R(x)$ can be written uniquely as a reduced fraction: 
	\begin{equation*}
		R(x)=\frac{s(x)}{q(x)},
		\end{equation*}
	where $q(x)\in \F_p[x]$ is monic and  $\gcd(s(x),q(x))=1$.	Moreover, $q(x)$ divides $m_p(x)$,  the $p$-main polynomial of $G$.
	\end{lemma}
\begin{proof}
		Let $f(x) = \sum_{i=0}^d c_i x^i$. Multiplying the formal power series $R(x)$ by $f(x)$ yields:
	\begin{equation*}
		f(x) R(x) = \left( \sum_{i=0}^d c_i x^i \right) \left( \sum_{j=0}^\infty \frac{e^\top A^j e}{x^{j+1}} \right) = \sum_{i=0}^d \sum_{j=0}^\infty c_i (e^\top A^j e) x^{i-j-1}.
	\end{equation*}
	 Thus, for $k\ge 0$, the coefficient of $x^{-(k+1)}$ in this product is exactly
	\begin{equation*}
		\sum_{i=0}^d c_i (e^\top A^{k+i} e), \text{~or equivalently,~}    e^\top A^{k}f(A) e.
	\end{equation*}
	Therefore, $f(x)R(x)$ contains no negative powers of $x$, that is, $f(x)R(x)$ is a polynomial, if and only if  
	\begin{equation}\label{efe}
		e^\top A^{k}f(A) e=0  \text{~for all~} k\ge 0, \text{~over~} \F_p.
	\end{equation}
	As $m_p(A)e=0$, we see that Eq.~\eqref{efe} holds when $f(x)=m_p(x)$.  This gives $R(x)=\frac{\tilde{s}(x)}{m_p(x)}$ for some $\tilde{s}(x)\in \F_p[x]$.
	Let
	\begin{equation*}
	s(x)=\frac{\tilde{s}(x)}{\gcd(\tilde{s}(x),m_p(x))}\text{~and~}
	q(x)=\frac{m_p(x)}{\gcd(\tilde{s}(x),m_p(x))}.
	\end{equation*} Then we have the desired form $R(x)=\frac{s(x)}{q(x)}$. This completes the proof of Lemma \ref{rsq}.
\end{proof}
In the following,  the symbol $q(x)$ always refers to the denominator polynomial of $R(x)$  as claimed in Lemma \ref{rsq}. From the proof, the polynomial $q(x)$ can be equivalently defined to be the monic polynomial $f(x)\in \F_p[x]$ of the smallest degree such that Eq.~\eqref{efe} holds.
	\begin{lemma} \label{pG}
	 We have $\Phi_p(G; x)=\frac{\chi(A;x)}{q(x)}$.
	\end{lemma}
	\begin{proof}
		As $(xI-A)^{-1}=\frac{\adj(xI - A) }{\chi(A; x)}$, we have 
		\begin{equation*}
			R(x)=e^\top (xI-A)^{-1}e=\frac{e^\top\adj(xI-A)e}{\chi(A;x)}.
		\end{equation*}
		By Lemma \ref{rsq}, we have $R(x)=\frac{s(x)}{q(x)}$ in reduced fraction. It follows that 
		\begin{equation}\label{g1}
			\gcd(e^\top\adj(xI-A)e,\chi(A;x))=\frac{\chi(A;x)}{q(x)}.
		\end{equation}
		On the other hand, by Lemma \ref{mdl}, we obtain
		\begin{equation*} 
			\chi(A+J; x) = \det(xI - A - ee^\top) = \chi(A; x) - e^\top \adj(xI - A) e.
		\end{equation*}
Consequently,
	\begin{equation}\label{g2}
		\Phi_p(G;x)=\gcd(\chi(A;x),\chi(A+J;x))=\gcd(e^\top\adj(xI-A)e,\chi(A;x)).
	\end{equation}
	Combining Eqs.~\eqref{g1} and \eqref{g2}, Lemma \ref{pG} follows.
	\end{proof}
	Recall that for a subspace $V$ of $\F_p^n$ (equipped with standard inner product $\langle u,v\rangle=u^\top v$), the \emph{orthogonal complement} $V^\perp$ is the subspace consisting of all vectors in $\F_p^n$ that are orthogonal to $V$, i.e., $V^\perp=\{u\in \F_p^n\colon\,u^\top v=0\text{~for all~} v\in V\}$. The \emph{radical} \cite[p.~266]{roman} of $V$, denoted by $\rad(V)$, is the intersection of $V$ and $V^\perp$. Note that $(V^\perp)^\perp=V$ and hence $\rad(V)=\rad(V^\perp)$.
	\begin{lemma}[\cite{wang2023EJC}]\label{vvp}
	Let $V\subset \F_p^n$ be an $A$-invariant subspace. Then $V^\perp$ is also $A$-invariant and  $\chi(A;x)=\chi(A|_V;x)\chi(A|_{V^\perp};x)$.
	\end{lemma}
	Now we are ready to prove Eq.~\eqref{fac}, which we restate as the following proposition for convenience. 
\begin{proposition}\label{bf}We have the factorization
		\begin{equation*}
		\Phi_p(G;x)=\chi(A|_{\N(W^\top)};x)\chi(A|_{\rad \N(W^\top)};x).
	\end{equation*} 
\end{proposition}
\begin{proof}
	Let $V=\col W=\operatorname{span} \{e,Ae,A^2e,\ldots\}$ and $U=\operatorname{span}\{q(A)e,Aq(A)e,A^2q(A)e,\ldots\}$. Clearly, $U\subset V$ and both are $A$-invariant. Noting that the left null space of $W$ is exactly the orthogonal complement of the column space of $W$, i.e., $\N(W^\top)=(\col W)^\perp=V^\perp$,  we have 
	\begin{equation*}
		\rad \N(W^\top)=\rad (V^\perp)=\rad V.
	\end{equation*}
	We claim that $\rad V=U$. Let $w$ be any vector in $\rad V$. Then $w\in V$ and hence there exists a polynomial $f\in\F_p[x]$ such that $w=f(A)e$. Since $w\in V^\perp$, we see that 
		\begin{equation}\label{efe2}
		e^\top A^{k}f(A) e=	e^\top A^{k}w=0  \text{~for all~} k\ge 0, \text{~over~} \F_p.
	\end{equation}
	Recall that $q(x)$ is the monic polynomial $f(x)\in \F_p[x]$ of the smallest degree such that Eq.~\eqref{efe2} holds. We obtain that $f(x)$ must be a multiple of $q(x)$. This implies $f(A)e\in U$, i.e., $w\in U$. Thus $\rad V\subset U$ by the arbitrariness of $w$. Conversely, as $e^\top A^k q(A)e=0$ for all $k\ge 0$, we easily see that $U\subset V^\perp$. Noting $U\subset V$, we have the reversed inclusion $U\subset \rad V$ and hence the claim follows.

	Note that for a cyclic subspace $S$ generated by some vector $\alpha$, i.e., $S=\operatorname{span}\{\alpha,A\alpha,A^2\alpha,\ldots\}$, the characteristic polynomial of $A$ restricted to $S$ is the unique monic polynomial $g(x)$ of the minimum degree such that $g(A)\alpha=0$. Thus, by Definition \ref{pm}, we have $\chi(A|_V;x)=m_p(x)$.

	Moreover, by Lemma \ref{rsq}, we know that $q(x)\mid m_p(x)$. Since $U$ is a cyclic subspace generated by $q(A)e$, its characteristic polynomial is the monic polynomial $h(x)$ of the minimum degree such that $h(A)q(A)e=0$. Since $m_p(A)e=0$, we easily find that $h(x) = \frac{m_p(x)}{q(x)}$, which gives $\chi(A|_U;x)=\frac{m_p(x)}{q(x)}$. 
	
	It follows from Lemma \ref{vvp} that $\chi(A|_{V^\perp};x)=\frac{\chi(A;x)}{\chi(A|_V;x)}=\frac{\chi(A;x)}{m_p(x)}$ and hence
	\begin{align*}
	\chi(A|_{\N(W^\top)};x)\chi(A|_{\rad \N(W^\top)};x)&=	\chi(A|_{V^\perp};x)\chi(A|_{\rad (V^\perp)};x)\\
	&=\chi(A|_{V^\perp};x)\chi(A|_{U};x)\\
	&=\frac{\chi(A;x)}{m_p(x)}\times\frac{m_p(x)}{q(x)}\\
	&=\frac{\chi(A;x)}{q(x)}\\
	&=\Phi_p(G;x),
	\end{align*}
	where we use  Lemma \ref{pG} for the last equality. This completes the proof of Proposition \ref{bf}.	
\end{proof}
\begin{corollary} \label{sqrtmid}
The polynomial	 $\sqrtp \Phi_p(G; x)$ divides $\chi(A|_{\N(W^\top)}; x)$ and hence 
\begin{equation*}
	\deg \sqrtp \Phi_p(G;x)\le \nullity_p W.
\end{equation*}
\end{corollary}
	\begin{proof}
		Let $g(x)=\chi(A|_{\N(W^\top)};x)$ and $h(x)=\chi(A|_{\rad \N(W^\top)};x)$. Since $\rad \N(W^\top)$ is a subspace of $\N(W^\top)$, we obtain $h(x)\mid g(x)$.  This, together with Proposition \ref{bf}, implies that
		\begin{equation*}
			\Phi_p(G;x)\mid g^2(x).
			\end{equation*}
		Now, let
		\begin{equation*}
		\Phi_p(G;x) = \prod_{1\le i\le r}f_i^{e_i}
		\end{equation*}
		be the irreducible factorization of $\Phi_p$, with distinct monic irreducible polynomials $f_1,f_2,\ldots, f_r$ and positive integers $e_1,e_2,\ldots, e_r$. Then it is easy to see that each $f_i$ must be a factor of $g(x)$ with multiplicity at least $\lceil\frac{e_i}{2}\rceil$, i.e., 
		\begin{equation*}
			\prod_{1\le i\le r}f_i^{\lceil\frac{e_i}{2}\rceil}\text{~divides~} g(x).
		\end{equation*}
		This completes the proof.
	\end{proof}
	\begin{corollary}\label{em}
		The polynomial $\frac{\chi(G;x)}{\Phi_p(G;x)}$ divides $m_p(G;x)$, and $m_p(G;x)$ divides  $\frac{\chi(G;x)}{\sqrtp\Phi_p(G;x)}$.
	\end{corollary}
	\begin{proof}
		By Lemma \ref{pG}, we have $\frac{\chi(G;x)}{\Phi_p(G;x)}=q(x)$. By Lemma \ref{rsq}, we obtain  $q(x)\mid m_p(G;x)$ and hence the first assertion of Corollary \ref{em} follows. Moreover, from the proof of Proposition \ref{bf}, we know that 
	\begin{equation*}
		m_p(G;x)=\frac{\chi(A;x)}{\chi(A|_{(\col W)^\perp};x)}=\frac{\chi(A;x)}{\chi(A|_{\N(W^\top)};x)}.
		\end{equation*}
		Thus, the second assertion clearly follows from  Corollary \ref{sqrtmid}.
	\end{proof}
	As a byproduct, we get the following interesting inequality on nullity of walk matrices for generalized cospectral graphs.
	\begin{corollary}\label{ghe}
		If $G$ and $H$ are generalized cospectral, then for any prime $p$, 
		\begin{equation*}
			\frac{1}{2}\nullity_p(W(G))\le \nullity_p(W(H))\le 2	\nullity_p(W(G)).
		\end{equation*} 
	\end{corollary}
	\begin{proof}
	Using Corollary \ref{em}, we find that 
	\begin{equation}\label{eqs}
	\deg \frac{\chi(G;x)}{\Phi_p(G;x)}\le \deg m_p(G;x)\le 	\deg \frac{\chi(G;x)}{\sqrtp\Phi_p(G;x)}.
	\end{equation}	As $\deg m_p(G)=n-\nullity_p(W(G))$, Eq. \eqref{eqs} can be restated as
	\begin{equation*}
\deg \sqrtp\Phi_p(G;x)\le \nullity_p(W(G))\le 	\deg \Phi_p(G;x).
	\end{equation*}
By the definition of the square root of a polynomial, we easily see that \begin{equation*}
	\deg \sqrtp\Phi_p(G;x)\ge \frac{1}{2}\deg \Phi_p(G;x),
\end{equation*}
and consequently, 
	\begin{equation}\label{ieqG}
\frac{1}{2}	\deg\Phi_p(G;x)\le \nullity_p(W(G))\le 	\deg \Phi_p(G;x).
\end{equation}

	Noting that $\Phi_p(G;x)$ is invariant under generalized cospectrality, we have $\Phi_p(G;x)=\Phi_p(H;x)$. Thus, we also have
		\begin{equation}\label{ieqH}
	\frac{1}{2}	\deg\Phi_p(G;x)\le \nullity_p(W(H))\le 	\deg \Phi_p(G;x).
	\end{equation}
	Combining Eqs.\eqref{ieqG} and \eqref{ieqH}, Corollary \ref{ghe} follows.	
	\end{proof}
	\section{Completing the proof of Theorem \ref{main}}
	We shall prove a slightly stronger result, from which Theorem \ref{main} is a direct consequence. The main tool is  Corollary \ref{em} together with some  skills developed in \cite{qiu2023DM} and \cite{wang2023EJC}. We begin with some standard preliminaries and notions.
	
	A graph is \emph{controllable} \cite{godsil2012AC} if $W(G)$ is non-singular. An orthogonal matrix $Q$ is called \emph{regular} if $Qe=e$.  Clearly, for two graphs $G$ and $H$, if there exists a regular orthogonal matrix $Q$ such that $Q^\top A(G) Q=A(H)$, then they are generalized cospectral. It turns out that the converse is also true, which was first proved by Johnson and Newman \cite{johnson1980JCTB} and  later developed systematically in the generalized characterization of controllable graphs since the work of Wang and Xu \cite{wang2006EJC}.
	\begin{lemma}[\cite{johnson1980JCTB,wang2006EJC}]\label{mc}
		Let $G$ be a controllable graph  and $H$ be a graph generalized cospectral with $G$. Then there exists a unique	regular rational orthogonal matrix $Q$ such that $Q^\top A(G)Q = A(H)$. Moreover, this unique $Q$ satisfies
		$Q^\top W(G) = W(H)$ and  is therefore rational.	
	\end{lemma}
	\begin{definition}[\cite{wang2006EJC}]
		For a rational matrix $Q$, the \emph{level} of $Q$, denoted by $\ell(Q)$, is the least positive integer $k$ such that each entry of $kQ$ is integral.
	\end{definition}
	The following estimation on the level of a matrix is a simple application of Smith normal forms.
	\begin{lemma}\label{bl}
		Let $M$ be a non-singular integer matrix and $d$ be the last invariant factor of $M$. Then $dM^{-1}$ is integral, i.e., $\ell(M^{-1})\mid d$.
	\end{lemma}
Let $\textup{RO}_n(\mathbb{Q})$ denote the group of all rational regular orthogonal matrices.	For a controllable graph $G$ of order $n$, we define
	\begin{equation*}
		\mathcal{Q}(G)=\{Q\in \textup{RO}_n(\mathbb{Q})\colon\, Q^\top A(G)Q\text{~is a $(0,1)$-matrix}\}.
		\end{equation*}
		Note that $Q^\top A(G)Q$ is symmetric and has trace zero. Thus, if $Q^\top A(G)Q$ is a $(0,1)$-matrix, it must be an adjacency matrix of some graph.
		
For an integer $m$ and a prime $p$, the \emph{$p$-adic valuation} \cite{gouvea} of $m$, denoted by $v_p(m)$, is the largest integer $k$ such that $p^k\mid m$. 
The following definition is essentially due to Qiu et al. \cite{qiu2023DM}.
\begin{definition}\label{rw}
	Let $G$ be an $n$-vertex controllable graph and $p$ be a prime factor of $\det W(G)$. The \emph{$p$-reduced walk matrix}, denoted by $W_{(p)}(G)$, or simply $W_{(p)}$, is the matrix
	\begin{equation}\label{dwp}
		W_{(p)}(G)=\left[e,Ae,\ldots,A^{r-1}e,\frac{f_p(A)e}{p},\frac{Af_p(A)e}{p},\ldots,\frac{A^{n-1-r}f_p(A)e}{p}\right],
	\end{equation}
	where $r=\deg f_p(x)$ and $f_p(x)\in \mathbb{Z}[x]$ is a monic polynomial whose reduction over $\F_p$ is  $\frac{\chi(G; x)}{\sqrtp\Phi_p(G; x)}$. For uniqueness, we usually assume each coefficient of $f_p(x)$ is in  $\{0,1,\ldots,p-1\}$.
\end{definition}
We remark that if $p\nmid \det W(G)$, then $m_p(G;x)$ coincides with the reduction of $\chi(G;x)$ modulo $p$ and it follows from Corollary  \ref{em} that $f_p(x)\equiv \chi(G;x)\pmod{p}$. Thus, in this case, we have $r=n$  and hence the matrix $W_{(p)}(G)$ given in Eq. \eqref{dwp} is exactly the ordinary walk matrix $W(G)$. On the other hand, for the case  $p\mid \det W(G)$, we have $\deg m_p(G)<n$ and hence $\Phi_p(G;x)\neq 1$ by the first assertion of Corollary \ref{em}. Consequently, $\sqrtp \Phi_p(G;x)\neq 1$ and hence $r=\deg f_p(x)<n$. Thus, for $p\mid \det W(G)$, we must have $W_{(p)}(G)\neq W(G)$. 
\begin{proposition}\label{wi}
	The matrix $W_{(p)}(G)$ is integral for each controllable graph $G$ and each prime  $p$ with $p\mid \det W(G)$.
\end{proposition}
\begin{proof}
	 By Corollary \ref{em}, the reduction of $f_p(x)$  modulo $p$ is a multiple of $m_p(G;x)$. Hence $f_p(A)e\equiv 0\pmod{p}$. This clearly implies that $A^i f_p(A)e\equiv 0\pmod{p}$, i.e., $\frac{A^i f_p(A)e}{p}$ is an integer vector, for any nonnegative integer $i$. Thus, the $p$-reduced matrix $W_{(p)}(G)$ is integral, completing the proof.	 
\end{proof}
\begin{theorem}\label{gmain}
	Let $G$ be an $n$-vertex controllable graph.  Then $v_p(\ell(Q))\le v_p(d_n(W_{(p)}))$ for any $Q\in \mathcal{Q}(G)$ and any prime factor $p$ of $\det W$.
\end{theorem}
\begin{proof}
	Let $H$ be the graph with adjacency matrix $A(H)=Q^\top A(G)Q$. Then we have $$Q^\top (A(G))^k Q=(A(H))^k$$ and hence 
\begin{equation*}
Q^\top (A(G))^k e=Q^\top (A(G))^k Qe=(A(H))^k e,
\end{equation*}
 for any $k\ge 0$. Thus, $Q^\top W(G)=W(H)$. 
 
  Let $f_p(G;x),f_p(H;x)\in \mathbb{Z}[x]$ be monic polynomials with coefficients in $\{0,1,\ldots,p-1\}$ whose reductions over $\F_p$ are
 \begin{equation*}
 	\frac{\chi(G;x)}{\sqrtp\Phi_p(G;x)}\text{~and~} \frac{\chi(H;x)}{\sqrtp\Phi_p(H;x)},
 \end{equation*}
 respectively. 
 But since $G$ and $H$ are generalized cospectral, we must have $\Phi_p(G;x)=\Phi_p(H;x)$ (and $\chi(G;x)=\chi(H;x)$ of course). It follows that
 \begin{equation*}
 	f_p(G;x)\equiv f_p(H;x)\pmod{p}
 \end{equation*}
 and hence   $f_p(G;x)=f_p(H;x)$ by the assumption on the range of coefficients.
 
 Let $p$ be any prime factor of $\det W(G)$ (or equivalently $\det W(H)$ as $\det W(G)=\pm \det W(H)$). Using Proposition \ref{wi} for both $G$ and $H$, we find that both $W_{(p)}(G)$ and $W_{(p)}(H)$ are integral.  Since  $f_p(G;x)=f_p(H;x)$, it is easy to see from the equality $Q^\top W(G)=W(H)$ and Definition \ref{rw} that
 \begin{equation*}
Q^\top W_{(p)}(G)=W_{(p)}(H).
 \end{equation*}
Let $d_n$ be the last invariant factor of $W_{(p)}(G)$. Then it is easy to see from Lemma~\ref{bl} that $\ell(Q) \mid d_n$, and hence $v_p(\ell(Q)) \le v_p(d_n)$.  This completes the proof of Theorem \ref{gmain}. 
\end{proof}
Now we show that Theorem \ref{main} is an easy consequence of Theorem \ref{gmain}.

\begin{proof}[Proof of Theorem \ref{main}]
Let $H$ be any graph that is generalized cospectral with $G$. By Lemma \ref{mc}, there exists a unique $Q\in \textup{RO}_n(\mathbb{Q})$ such that $Q^\top A(G)Q=A(H)$. Let $\ell=\ell(Q)$ be the level of $Q$. Noting  that $d_n$ is even, the squarefreeness assumption clearly implies that $d_n\equiv 2\pmod{4}$.  Thus, $\ell$ is odd, according to \cite[Corollary 1]{wang2023EJC}. We further claim that $\ell=1$.

Suppose to the contrary that $\ell>1$. Since $\ell$ is odd, there exists an odd prime $p$ such that $p\mid \ell$.  As $Q^\top=W(H)(W(G))^{-1}$, we have $\ell\mid d_n$ by Lemma \ref{bl}, where $d_n$ is the last invariant factor of $W(G)$. Thus, $p\mid d_n$. Let $d_n^\prime$ be the last invariant factor of the $p$-reduced walk matrix $W_{(p)}(G)$. It follows from Theorem \ref{gmain} that $v_p(\ell)\le v_p(d_n^\prime)$. By Definition \ref{rw}, it is easy to see that
\begin{equation*}
v_p(\det W_{(p)}(G))=v_p(\det W(G))-\deg \sqrtp \Phi_p(G;x).
\end{equation*}
As $d_n$ is square-free, we know that $v_p(\det W(G))=\nullity_p W(G)$. Noting that $d_n^\prime$ is a factor of $\det W_{(p)}(G)$, we obtain \begin{align*}
v_p(d_n^\prime)&\le v_p(\det W_{(p)}(G))\\
&=v_p(\det W(G))-\deg \sqrtp \Phi_p(G;x)\\
&=\nullity_p W(G)-\deg \sqrtp \Phi_p(G;x)\\
&=0,
\end{align*} 
where we use Eq.~\eqref{keyequ2} for the last equality. Therefore $v_p(d_n^\prime)=0$ and hence $v_p(\ell)=0$, i.e., $p\nmid \ell$. This is  a contradiction and hence the claim follows.  This means that the rational regular orthogonal matrix $Q$ is indeed an integer matrix, which clearly implies that $Q$ is a permutation matrix. Thus, $H$ is isomorphic to $G$ and hence $G$ is DGS by the arbitrariness of $H$.
\end{proof}

	\section{On annihilating polynomial for $p=2$}
	In this section, we consider the case $p=2$. Following Definition \ref{rw}, let $f_2(G;x)$, or simply $f_2(x)$, be a monic polynomial whose reduction over $\F_2$ is $\frac{\chi(G;x)}{\sqrtp \Phi_2(G;x)}$. By the last assertion of Corollary \ref{em}, we know that, over $\F_2$,  the polynomial $f_2(x)$ is a multiple of $m_2(G;x)$, the 2-main polynomial of $G$. This means $f_2(x)$ is an annihilating polynomial of $A$ with respect to $e$ over $\F_2$, i.e., $f_2(A)e\equiv 0\pmod{2}$. In \cite{wang2017JCTB} (see also \cite{qiu2023DM}), Wang gave a straightforward  way to construct an annihilating polynomial over $\F_2$. The main aim of this section is to show that the two constructions  coincide, which we state as the following theorem.
	\begin{theorem}\label{samepoly}
		Let $G$ be an $n$-vertex graph whose characteristic polynomial is $\chi(G;x)=x^n+c_1x^{n-1}+\cdots+c_{n-1}x+c_n$. Then, 
		\begin{equation*}
	f_2(G;x)\equiv\begin{cases}
			x^{\frac{n}{2}}+c_2	x^{\frac{n-2}{2}}+\cdots+c_{n-2}x+c_n\pmod{2}&\text{if $n$ is even},\\
				x(x^{\frac{n-1}{2}}+c_2	x^{\frac{n-3}{2}}+\cdots+c_{n-3}x+c_{n-1})\pmod{2}&\text{if $n$ is odd.}
		\end{cases}
		\end{equation*}
	\end{theorem}
Theorem \ref{samepoly} shows that the square-root polynomial 
	provides a unified framework for both odd and even prime factors. To show Theorem \ref{samepoly}, 
	we first establish a simple connection between the characteristic polynomial of $A+J$ and the formal derivative of the characteristic polynomial over $\F_2$.
	\begin{lemma} \label{de}
		Let $A$ be the adjacency matrix of a simple graph. Then, over $\F_2$,
		\begin{equation*}
	\chi(A+J;x) = \chi(A;x) + \chi'(A;x).
		\end{equation*}
	\end{lemma}
	\begin{proof}
 Working over $\F_2$, we have $-A=A$, so $\chi(A;x)=\det(xI+A)$. By Jacobi's formula \cite{harville2008}, the formal derivative is \begin{equation*}
 \chi'(A;x) = \tr(\adj(xI+A)).
 \end{equation*}
 Since $xI+A$ is symmetric, its adjugate matrix $M = \adj(xI+A)$ is also symmetric.  Thus, over $\F_2$,  we have $$\tr (\adj (xI+A)) =e^\top (\adj (xI+A))e$$ and hence
\begin{equation}\label{e1}
			\chi'(A;x)=e^\top(\adj (xI+A))e.
\end{equation}
On the other hand, by  Lemma \ref{mdl}, we obtain
		\begin{equation}\label{e2}
		\chi(A+J;x)=\det(xI+A+ee^\top)=\chi(A;x)+e^\top(\adj (xI+A))e.
		\end{equation}
		Combining Eqs.~\eqref{e1} and \eqref{e2}, the lemma follows.
	\end{proof}
	
\begin{proof}[Proof of Theorem \ref{samepoly}]
We work on the field $\F_2$. Noting that $c_i=0$ over $\F_2$ for odd index $i$ by  Sachs' Coefficient Theorem \cite{cvetkovic2010}, we have
\begin{equation}\label{ch1}
	\chi(A;x)=\begin{cases}
			x^{n}+c_2	x^{n-2}+\cdots+c_{n-2}x^2+c_n&\text{if $n$ is even},\\
		x^{n}+c_2	x^{n-2}+\cdots+c_{n-3}x^3+c_{n-1}x&\text{if $n$ is odd.}
		\end{cases}
\end{equation} 
Let 
\begin{equation*}
	g(x)=\begin{cases}
		x^{\frac{n}{2}}+c_2	x^{\frac{n-2}{2}}+\cdots+c_{n-2}x+c_n&\text{if $n$ is even},\\
		x^{\frac{n-1}{2}}+c_2	x^{\frac{n-3}{2}}+\cdots+c_{n-3}x+c_{n-1}&\text{if $n$ is odd.}
	\end{cases}
\end{equation*}
Since $h^2(x)=h(x^2)$ holds for any $h(x)\in\F_2[x]$, we can rewrite  Eq.~\eqref{ch1} as
\begin{equation}\label{cg}
\chi(A;x)=\begin{cases}
	g^2(x)&\text{if $n$ is even},\\
	xg^2(x)&\text{if $n$ is odd.}
\end{cases}
\end{equation}
Consequently,  we have 
\begin{equation*}
	\chi'(A;x)=\begin{cases}
		0&\text{if $n$ is even},\\
		g^2(x)&\text{if $n$ is odd.}
	\end{cases}
\end{equation*}
It follows from Lemma \ref{de} that 
\begin{equation*}
	\chi(A+J;x)=\begin{cases}
		g^2(x)&\text{if $n$ is even},\\
		(1+x)g^2(x)&\text{if $n$ is odd}.
	\end{cases}
\end{equation*}
Thus we always have $\Phi_2(G;x)=\gcd(\chi(A;x),\chi(A+J;x))=g^2(x)$, which implies 
\begin{equation*}
\sqrtp \Phi_2(G;x)=g(x).
\end{equation*}
Combining with Eq.~\eqref{cg} leads to
\begin{equation*}
	\frac{\chi(G;x)}{\sqrtp \Phi_2(G;x)}=\begin{cases}
		g(x)&\text{if $n$ is even},\\
		xg(x)&\text{if $n$ is odd.}
	\end{cases}
\end{equation*}
This completes the proof of Theorem \ref{samepoly}.
\end{proof}
\section{Application to DGS-graph construction}
In a recent paper \cite{wang2026DAM}, Wang et al. proposed a general method to construct larger DGS-graphs using the rooted product operation.  Let $H^{(v)}$ be a rooted graph, where $v$ denotes the root vertex. For a graph $G$, the rooted product graph $G\circ H^{(v)}$ is obtained from $G$ and $n$ copies of $H$ by identifying the root of the $i$th copy of $H$  with the $i$th vertex of $G$ for each $i$. They \cite{wang2026DAM} focused on a specific family of graphs
\begin{equation*}
	\mathcal{F}_n=\{\text{an $n$-vertex graph~} G\colon\, \det A(G)=\pm 1\text{~and~}2^{-\lfloor\frac{n}{2}\rfloor}\det W(G)=\pm 1\}.
\end{equation*}
We note that $\mathcal{F}_n$ is empty when $n$ is odd as any adjacency matrix of  odd order is singular over $\F_2$ (being skew-symmetric with zero diagonal), whereas $\mathcal{F}_n$ is nonempty for all even integers $n\ge 6$, see \cite{liu2019DM}. Let $\mathcal F=\cup \mathcal{F}_n$, where $n$ runs through all even integers $n\ge 6$.
It is known that each graph in $\mathcal{F}$ is DGS by Theorem \ref{wt}.   In the same paper, Wang et al. obtained a sufficient condition for a rooted graph $H^{(v)}$ to be an \emph{ $\mathcal{F}$-preserver}, that is, to guarantee $G\circ H^{(v)}\in \mathcal{F}$ for each $G\in \mathcal{F}$. 

As an application of our main result (Theorem \ref{main}), we aim to generalize this preservation property to a larger family of graphs where the last invariant factor of $W$ is squarefree but may have odd prime factors.

	\begin{definition}\label{dg} Let $\mathcal{G}_n$ ($n$ even) be the family of all $n$-vertex graphs $G$  satisfying the following conditions:
	
	(i)  $\det A(G)=\pm 1$;
	
	(ii) the last invariant factor of $W(G)$ is square-free; and 
	
	(iii) $\deg \Phi_p(G;x)=2\nullity_p(W(G))$ for each odd prime factor $p$ of $\det W(G)$.
\end{definition} 

Let $\mathcal{G}=\cup \mathcal{G}_n$. We note that $\mathcal{F}\subset \mathcal{G}$. Indeed, for each graph $G\in \mathcal{F}$, the last invariant factor of $W(G)$ is 2 and hence clearly satisfies the last two conditions of Definition \ref{dg}. 

The following lemma gives some equivalent characterizations of the last condition of Definition \ref{dg}.

\begin{lemma}\label{equiv3}
	Let $G$ be an  $n$-vertex  controllable graph. Then for each odd prime factor of $\det W$, the following conditions are equivalent:
	
	\textup{(i)} $\deg \Phi_p(G;x)=2\nullity_p W$;
	
	\textup{(ii)} $\N(W^\top)=\rad \N(W^\top)$, i.e., each pair of vectors in $\N(W^\top)$ is orthogonal; and
	
	\textup{(iii)} $\Phi_p(G;x)$ is a perfect square in $\F_p[x]$ and $\deg \sqrtp \Phi_p(G;x)=\nullity_p W$.
\end{lemma}

\begin{proof}
	By Lemma \ref{bf}, we have	\begin{equation}\label{fac3}
		\Phi_p(G;x)=\chi(A|_{\N(W^\top)};x)\chi(A|_{\rad \N(W^\top)};x),
	\end{equation} 
	which implies 
\begin{equation*}
\deg \Phi_p(G;x)=\dim \N(W^\top)+\dim \rad\N(W^\top)\le 2\dim \N(W^\top)= 2\nullity_p W.
	\end{equation*}
Noting that the non-strict inequality becomes an equality if and only if  $\N(W^\top)=\rad \N(W^\top)$, the equivalence of (i) and (ii) follows.  Now, suppose (ii) holds. Then Eq.~\eqref{fac3}  reduces to 
\begin{equation*}
	\Phi_p(G;x)=(\chi(A|_{\N(W^\top)};x))^2.
\end{equation*} 
Thus, $\Phi_p(G;x)$ is a perfect square and 
 $\deg  \sqrtp \Phi_p(G;x)=\deg \chi(A|_{\N(W^\top)};x)=\nullity_p W$. This proves (ii)$\implies$(iii). Finally, if (iii) holds, then we have 
 $$\deg \Phi_p(G;x)=2\deg \sqrtp \Phi_p(G;x)=2\nullity_p W.$$ Thus, (iii)$\implies$(i). This completes the proof.	
\end{proof}
The following corollary is immediate from Theorem \ref{main} and Lemma \ref{equiv3}. 
\begin{corollary}
Each graph in $\mathcal{G}$ is DGS.
\end{corollary}
Let $H^{(v)}$ be a rooted graph of order $m$ with root $v$. Let $B(\lambda)=A(H)+\lambda D_v$, where $D_v=\diag(0,\ldots,0,1,0\ldots,0)$ with exactly one nonzero entry (which is 1) corresponding to the root $v$. For two polynomials $f(x)$ and $g(x)$, we use $\Res(f(x),g(x))$ to denote the resultant of $f(x)$ and $g(x)$. We need the following two formulas of  $\det A(G\circ H^{(v)})$ and $\det W(G\circ H^{(v)})$.
\begin{lemma}[\cite{wang2026DAM}]\label{da}
	$\det A(G\circ H^{(v)})=\det((\det A(H))\cdot I_n+(\det A(H-v))\cdot A(G)).$
\end{lemma}
\begin{lemma}[\cite{wang2026DAM}]\label{dw}
We have
\begin{equation*}
	\det W(G\circ H^{(v)})=\pm \left(\Res(\chi(H;x),\chi(H-v;x))\right)^{\frac{n(n-1)}{2}}\cdot\det h(A(G))\cdot\left(\det W(G)\right)^m,
\end{equation*}
where $h(\cdot)$ is a polynomial defined by $h(\lambda)=\det[e,B(\lambda)e,\ldots,(B(\lambda))^{m-1}e]$. 
\end{lemma}

Following \cite{wang2026DAM}, for a rooted graph $H^{(v)}$, we consider the following three  conditions:
\begin{itemize}
	\item[(C1)] either $\det A(H)=\pm 1$ and $\det A(H-v)=0$, or $\det A(H)=0$ and $\det A(H-v)=\pm 1$;
	\item[(C2)] $\Res(\chi(H;x), \chi(H-v;x))=\pm 1$;
	\item[(C3)] $\det[e, B(\lambda)e, \ldots, (B(\lambda))^{m-1}e]=\pm \lambda^{k}$ for some $k\ge 0$.
\end{itemize}

As a direct application of Lemmas \ref{da} and \ref{dw}, Wang et al. \cite{wang2026DAM} showed that each rooted graph $H^{(v)}$ satisfying the above three conditions is necessarily an $\mathcal{F}$-preserver. The main aim of this section is to show that $H^{(v)}$ is also a $\mathcal{G}$-preserver under the same conditions. Since $\mathcal{G}$ is a larger DGS-family than $\mathcal{F}$, it gives many new constructions of DGS-graphs.
\begin{theorem}\label{gp}
	Let $H^{(v)}$ be a rooted graph satisfying (C1), (C2) and (C3), simultaneously. Then $H^{(v)}$ is an $\mathcal{G}$-preserver, i.e., $G\circ H^{(v)}\in \mathcal{G}$ for any $G\in\mathcal{G}$. 
\end{theorem}
\begin{lemma} \label{wak}
	Let $W$ and $\tilde{W}$ be the walk matrices of $G$ and $G\circ H^{(v)}$, respectively. Let $p$ be an odd prime factor of $\det W(G)$.  If $\alpha\in \N(W^\top) \subset \F_p^n$, then $u\otimes \alpha\in \N(\tilde{W}^\top)$ for any $u\in \F_p^m$. \end{lemma}
\begin{proof}
	By labeling the vertices of $G\circ H^{(v)}$ appropriately, we may write
	\begin{equation*}
		A(G\circ H^{(v)})=A(H)\otimes I_n+D_v\otimes A(G).
	\end{equation*}
We regard	$A(G\circ H^{(v)})$ as a $m\times m$ block matrix $(C_{ij})$ with each block  $C_{ij}\in \{A(G),I,O\}$, where $O$ is the zero matrix (of order $n$). Let $A=A(G)$ and $\tilde{A}=A(G\circ H^{(v)})$. As $A$, $I$ and  $O$ are clearly pairwise commutative, we find from the product of block matrices that for each $k\ge 0$, the $k$-th power $\tilde{A}$ has the form 
\begin{equation}\label{ak}
	\tilde{A}^k=(f_{ij}(A))_{m\times m}
\end{equation}
	where $f_{ij}$ (depending on $k$) is a polynomial  with integer coefficients and has degree at most $k$.
Thus, using the form of Kronecker products, we may rewrite  Eq.~\eqref{ak} as
\begin{equation*}
\tilde{A}^k = \sum_{r=0}^k M_{k,r} \otimes A^r.
\end{equation*} Since $\alpha \in \N(W^\top)$, we have $\alpha^\top A^r e_n = 0$ for all $r \ge 0$. Using some standard properties of Kronecker products, we obtain
	\begin{equation*}
		e_{mn}^\top \tilde{A}^k (u \otimes \alpha) = (e_m^\top \otimes e_n^\top) \left( \sum_{r=0}^k M_{k,r} \otimes A^r \right) (u \otimes \alpha) = \sum_{r=0}^k (e_m^\top M_{k,r} u) \left( e_n^\top A^r \alpha \right) = 0.
	\end{equation*}
This implies $(u \otimes \alpha)^\top \tilde{W} = 0$, i.e., $u \otimes \alpha\in \N(\tilde{W}^\top)$. The proof is completed.
\end{proof}
\begin{proof}[Proof of Theorem \ref{gp}] We prove the theorem by establishing that $G\circ H^{(v)}$ satisfies the three conditions of Definition \ref{dg}. By Lemma \ref{da} and the assumptions on the determinants of three matrices $A(G)$, $A(H)$, and $A(H-v)$, we easily obtain that $\det A(G\circ H^{(v)})=\pm 1$. We proceed to check the second condition of Definition \ref{dg} for the rooted product $G\circ H^{(v)}$. 
	
	Let $d$ and $\tilde{d}$ be the last invariant factors of $W=W(G)$ and $\tilde{W}=W(G\circ H^{(v)})$. As $d$ is square-free, we have $v_2(d)=1$ and hence $v_2(\det W)=\lfloor{\frac{n}{2}}\rfloor=\frac{n}{2}$ as $n$ is even; see \cite[Theorem 16]{wang2021LAA}.  Let
	\begin{equation*}
		\det W=\pm 2^{\frac{n}{2}}p_1^{t_1}p_2^{t_2}\cdots p_s^{t_s}
	\end{equation*}
	be the factorization of $\det W$, where $p_i$'s are distinct odd primes and $t_i$'s are positive integers. It follows from Lemma \ref{dw}, together with the conditions on $G$ and $H^{(v)}$, that 
		\begin{equation}\label{dhw}
		\det \tilde{W}=\pm\left(\det W(G)\right)^m=\pm 2^{\frac{mn}{2}}p_1^{mt_1}p_2^{mt_2}\cdots p_s^{mt_s}.
	\end{equation}
Since at least $\frac{mn}{2}$ invariant factors of $\tilde{W}$ are even,  the fact $v_2(\det \tilde{W})=\frac{mn}{2}$ from Eq.~\eqref{dhw}  implies that $v_2(\tilde{d})=1$. Thus, to show that $\tilde{d}$ is square-free, it suffices to show $v_{p_i}(\tilde{d})=1$ for each $i\in\{1,2,\ldots,s\}$.

For simplicity of notation, we fix an $i$ in $\{1,2,\ldots,s\}$ and write $p=p_i$ and $t=t_i$. As $v_p(d)=1$ and $v_p(\det W)=t$, we see that exactly $t$ invariant factors are multiples of $p$, i.e., $\nullity_p(W)=t$. 
Let $V=\F_p^m\otimes \N(W^\top)$. Then we have 
\begin{equation*}
\dim V=\dim \F_p^m\times \dim \N(W^\top)=m\times\nullity_p W=mt.
\end{equation*}
By 	Lemma \ref{wak}, we find that $\N(\tilde{W}^\top)\supset \F_p^m\otimes \N(W^\top)=V$ and hence $\dim \N(\tilde{W}^\top)\ge mt$. On the other hand, by Eq.~\eqref{dhw}, we see that $v_p(\det \tilde{W})=mt$ and hence $\nullity_p \tilde{W}\le mt$ with equality holding if and only if $v_p(\tilde{d})=1$. Noting that $\dim \N(\tilde{W})=\nullity_p \tilde{W}$, we must have $\nullity_p \tilde{W}=mt$ and hence $v_p(\tilde{d})=1$. This proves that $\tilde{d}$ is square-free. Furthermore, as $\dim \N(\tilde{W}^\top)\ge mt=\dim V$ and $\N(\tilde{W}^\top)\supset V$, we see that the two spaces  $\N(\tilde{W}^\top)$ and $V$ are equal.

We claim that any two vectors in the subspace $V\subset \F_p^{mn}$ are orthogonal.  	Take any two vectors $x = u \otimes \alpha$ and $y = w \otimes \beta$ in $V$. Their inner product is $x^\top y = (u^\top w) (\alpha^\top \beta)$. By Lemma \ref{equiv3}(ii), we have $\alpha^\top \beta = 0$. Thus $x^\top y = 0$ and hence the claim follows by bilinearity.  Note that $\N(\tilde{W}^\top)=V$. It follows from Lemma \ref{equiv3} that $\deg \Phi_p(G\circ H^{(v)};x)=2\nullity_p\tilde{W}$, verifying the last condition of Definition \ref{dg} for the graph $G\circ H^{(v)}$. This completes the proof of Theorem \ref{gp}.
\end{proof}
\begin{example} \normalfont
Let  $G$ be an $8$-vertex graph whose adjacency matrix is
\begin{equation*}
	A=\begin{pmatrix}
		0 & 1 & 0 & 0 & 1 & 1 & 1 & 0 \\
		1 & 0 & 1 & 0 & 1 & 1 & 0 & 1 \\
		0 & 1 & 0 & 0 & 1 & 0 & 0 & 0 \\
		0 & 0 & 0 & 0 & 1 & 1 & 0 & 0 \\
		1 & 1 & 1 & 1 & 0 & 0 & 1 & 1 \\
		1 & 1 & 0 & 1 & 0 & 0 & 1 & 1 \\
		1 & 0 & 0 & 0 & 1 & 1 & 0 & 0 \\
		0 & 1 & 0 & 0 & 1 & 1 & 0 & 0 \\
	\end{pmatrix}.
\end{equation*}
A direct calculation shows that  $\det A(G)=1$, $\det W=-2^4\times 3$. Moreover, for the odd prime $p=3$, we obtain $\Phi_3(G;x)=(x+1)^2$ and hence the condition $\deg\Phi_3(G;x)=2\nullity_3 W(G)$ is verified.  Thus, $G$ belongs to the family $\mathcal{G}$ introduced in Definition \ref{dg}. 

Let $P_m$ be the path of order $m$ whose vertices are labeled naturally as $1,2,\ldots,m$. Consider a sequence of graphs $\{G_m\}$ defined by $G_m=G\circ P_m^{(1)}$.  It has been shown essentially \cite{wang2024LMA} that $P_m^{(1)}$ satisfies (C1), (C2) and (C3) simultaneously. Thus, by Theorem \ref{gp}, each graph $G_m$ belongs to $\mathcal{G}$ and hence is DGS.  Notably, the DGS-property of some $G_m$ relies heavily on Theorem \ref{main}, whereas the stronger condition of Theorem \ref{main0} is not satisfied.

Consider the graph $G_4=G\circ P_4^{(1)}$. It can be computed that the SNF of $W(G_4)$ is
\begin{equation*}
	\diag[\underbrace{1,\ldots,1}_{16},\underbrace{2,\ldots,2}_{12},\underbrace{6,6,6,6}_4].
\end{equation*}
Furthermore, $\Phi_3(G_4;x)=(1+x)^8$, a perfect square in $\F_3[x]$. Consequently, we obtain 
\begin{equation*}
\deg\sqrtp\Phi_3(G_4;x)=4=\nullity_3(W(G_4))
\end{equation*}
whereas
\begin{equation*}
	\deg\sfp\Phi_3(G_4;x)=1<\nullity_3(W(G_4)).
\end{equation*}
This means that the graph $G_4$ satisfies the condition of Theorem \ref{main}, but does not satisfy the condition of Theorem \ref{main0}.  
\end{example}
\section*{Declaration of competing interest}
There is no conflict of interest.
\section*{Acknowledgments}
This work is  supported by the National Natural Science Foundation of China (Grant Nos.~12001006 and 12371357) and  Wuhu Science and Technology Project, China (Grant No.~2024kj015).

\end{document}